\documentclass[sn-mathphys,Numbered]{sn-jnl}

\usepackage{graphicx}%
\usepackage{multirow}%
\usepackage{amsmath,amssymb,amsfonts}%
\usepackage{amsthm}%
\usepackage{mathrsfs}%
\usepackage[title]{appendix}%
\usepackage{xcolor}%
\usepackage{textcomp}%
\usepackage{manyfoot}%
\usepackage{booktabs}%
\usepackage{algorithm}%
\usepackage{algorithmicx}%
\usepackage{algpseudocode}%
\usepackage{listings}%

\theoremstyle{thmstyleone}%
\newtheorem{theorem}{Theorem}
\newtheorem{proposition}[theorem]{Proposition}%

\theoremstyle{thmstyletwo}%
\newtheorem{remark}{Remark}%
\newtheorem{corollary}{Corollary}%
\newtheorem{lemma}{Lemma}%

\theoremstyle{thmstylethree}%

\begin{document}

\title[Article Title]{Epi-convergence in distribution of normal integrands with applications to sets of $\epsilon$-optimal solutions}


\author{\fnm{Dietmar} \sur{Ferger}}\email{dietmar.ferger@tu-dresden.de}



\affil{\orgdiv{Fakult\"{a}t Mathematik}, \orgname{Technische Universit\"{a}t Dresden}, \orgaddress{\street{Zellescher Weg 12-14}, \city{Dresden}, \postcode{01069}, \country{Germany}}}




\abstract{We derive necessary and sufficient conditions for epi-convergence in distribution of normal integrands.
As a basic tool for the proof a new characterisation for distributional convergence of random closed sets is used. Our approach via the
epi-topology allows us to show that, if a net of normal integrands epi-converges in distribution, then the pertaining sets of $\epsilon$-optimal solutions converge in distribution in the underlying hyperspace endowed with the upper-Fell topology. Under some boundedness and uniquenss assumptions the convergence even holds for the Fell topology. Finally, measurable selections converge weakly to a Choquet-capacity.
}




\maketitle

\section{Introduction}
Let $(E,\mathcal{G})$ be a locally compact second countable Hausdorff-space (lcscH) with $\mathcal{F}$ and $\mathcal{K}$  the pertaining families of
all closed sets and all compact sets, respectively. $E$ is called the \emph{carrier space}. We want to equip $\mathcal{F}$ with some topology. For that purpose introduce for every subset $A \subseteq E$ the systems
$\mathcal{M}(A):=\{F \in \mathcal{F}: F \cap A = \emptyset\}$ of all \emph{missing sets}
and
$\mathcal{H}(A):=\{F \in \mathcal{F}: F \cap A \neq \emptyset\}$ of all \emph{hitting sets} of $A$.
Put
$$
\mathcal{S}:= \{\mathcal{M}(K): K \in \mathcal{K}\} \cup \{\mathcal{H}(G): G \in \mathcal{G}\} \subseteq 2^\mathcal{F}.
$$
Then the topology  on $\mathcal{F}$ generated by $\mathcal{S}$ is called \emph{Fell-topology} and denoted by $\tau_F$. It goes back to J. Fell \cite{Fell}.
Convergence in the Fell-topology is the same as convergence in the sense of Painlev\'{e}-Kuratowski, confer Theorem C.7 in Molchanov \cite{Molchanov}.
It induces the  Borel-$\sigma$ algebra $\underline{\mathcal{B}}_F := \sigma(\tau_F)$, the smallest $\sigma$-algebra on $\mathcal{F}$ containing
the Fell-topology. If $(\Omega,\mathcal{A}, \mathbb{P})$ is some probability space, then by definition a \emph{random closed set (in $E$ on $\Omega$)} is a map
$C:\Omega \rightarrow \mathcal{F}$, which is $\mathcal{A}-\underline{\mathcal{B}}_F$ measurable. Occasionally, we will write $\mathcal{F} \equiv \mathcal{F}(E)$ and $\tau_F \equiv \tau_F(E)$
in order to underline the basic carrier space $E$.

Besides random closed sets we will investigate
so-called \emph{normal integrands}. To explain this notion first consider the collection $S \equiv S(E)$ of all lower semicontinuous (lsc) functions
$f:E \rightarrow \overline{\mathbb{R}}$, where $\overline{\mathbb{R}}=[-\infty,\infty]$ is the extended real line:
$$
 S :=\{f:E \rightarrow \overline{\mathbb{R}}; f \text{ lsc}\}.
$$
Similarly as for $\mathcal{F}$ the function space $S$ will be endowed with a topology. With this in mind consider for
each $A \subseteq E$ the functional $I_A:S \rightarrow \overline{\mathbb{R}}$ defined by
$$
 I_A(f):= \inf_{x \in A} f(x).
$$
Let the \emph{epi-topology} $\tau_e$ be the coarsest topology on $S$ with respect to which $I_K$ is lsc for all $K \in \mathcal{K}$ and
$I_G$ is upper semicontinuous (usc) for all $G \in \mathcal{G}$. Convergence in the epi-topology is equivalent to epi-convergence. This follows from Theorem 5.3.2 in Molchanov \cite{Molchanov}. If $\mathcal{B}_e:= \sigma(\tau_e)$ is the corresponding Borel-$\sigma$ algebra
on $S$, then a mapping $Z:\Omega \rightarrow S$, that is $\mathcal{A}-\mathcal{B}_e$ measurable is called \emph{normal integrand (on $\Omega$)}.
According to Lemma 2.5 of Ferger \cite{Ferger1} our notion of normal integrand coincides with that of Molchanov \cite{Molchanov} and Rockafellar and Wets \cite{RockWets}.

Given a directed set $(A, \le)$ and a family of probability spaces $(\Omega_\alpha, \mathcal{A}_\alpha,\mathbb{P}_\alpha), \alpha \in A,$ we consider
random closed sets $C_\alpha$ in $E$ and normal integrands $Z_\alpha$ on $\Omega_\alpha$. In our paper we first give an equivalent condition for
distributional convergence of a net $(C_\alpha)_{\alpha \in A}$ to a limit $C$:
\begin{equation} \label{dconvrcs}
 C_\alpha \stackrel{\mathcal{D}}{\longrightarrow} C \quad \text{in } (\mathcal{F},\tau_F).
\end{equation}
Notice that the convergence in (\ref{dconvrcs}) by definition is the same as convergence of the
distributions $Q_\alpha := \mathbb{P}_\alpha \circ C_\alpha^{-1}$ to $\mathbb{P} \circ C^{-1} =: Q$ in the weak topology:
\begin{equation} \label{weakrcs}
 Q_\alpha \longrightarrow_w Q \quad \text{on } (\mathcal{F},\tau_F).
\end{equation}
It follows from the definition of Tops{\o}e \cite{Top} or G\"{a}nssler and Stute \cite{Stute} that weak convergence (\ref{weakrcs}) holds for instance if and only if
$$
 \liminf_\alpha Q_\alpha(\textbf{O}) \ge Q(\textbf{O}) \quad \forall \; \textbf{O} \in \tau_F.
$$
A short introduction of the weak topology and weak convergence of probability measures on arbitrary topological spaces can be found in Ferger \cite{Ferger2}.
Once the above announced criterion for (\ref{dconvrcs}) or (\ref{weakrcs}), respectively, is available it will be used to describe
distributional convergence of $(Z_\alpha)$ to $Z$, i.e.
$$
 Z_\alpha \stackrel{\mathcal{D}}{\longrightarrow} C \quad \text{in } (S,\tau_e)
$$
or what is the same
$$
  P_\alpha \longrightarrow_w P \quad \text{on } (S,\tau_e),
$$
where $P=\mathbb{P}\circ Z^{-1}$ and $P_\alpha=\mathbb{P}_\alpha \circ Z_\alpha^{-1}$ are the distributions of $Z$ and $Z_\alpha$ under $\mathbb{P}$ and $\mathbb{P}_\alpha$, respectively. This is possible, because
there is a strong relation between the two spaces $(\mathcal{F},\tau_F)$ and $(S,\tau_e)$, which can be expressed by \emph{epigraphs}. Here, for any
function $f:E \rightarrow \overline{\mathbb{R}}$ the \emph{epigraph} of $f$ is the set
$$
 \text{epi}(f):=\{(x,a) \in E \times \mathbb{R}: f(x) \le a \}.
$$
It is well-known that the function $f$ is lsc if and only if epi$(f)$ is a closed subset of $E \times \mathbb{R}$ endowed with the product-topology.
Let $\mathcal{E}=\{\text{epi}(f): f \in S\} \subseteq \mathcal{F}(E \times \mathbb{R})$ be the family of all epigraphs and $\sigma$ be the subspace topology on $\mathcal{E}$. Then Attouch \cite{Attouch}, p.254-255, shows that $(\mathcal{E},\sigma)$ is compact and that    the map
\begin{equation} \label{phi}
\phi:(S,\tau_e) \rightarrow (\mathcal{E},\sigma) \text{ defined by } \phi(f)=\text{epi}(f),
\end{equation}
is a homeomorphism.\\

The paper is organized as follows: In the next section we derive a new necessary and sufficient condition for weak convergence of probability measures on the hyperspace $\mathcal{F}$ equipped with the Fell topology. This result is then used in section 3 to find a new equivalent characterization for weak convergence
of probability measures on the space $S$ of all lower semicontinuous functions endowed with the epi-topology. In section 4 we show that epi-convergence in distribution of normal integrands entails the convergence of the corresponding $\epsilon$-optimal solutions as random closed sets.
In particularly, it follows that single solutions converge in distribution to the almost sure unique minimizing point of the limit normal integrand.
If uniqueness is not given, then the solutions converge to the entire set, say $C$, of all minimizers. The latter means that the distributions converge weakly to the capacity-functional of $C$.

\section{Weak convergence of probability measures on $(\mathcal{F},\tau_F)$}
In this section we give a necessary and sufficient condition for $Q_\alpha \longrightarrow_w Q \quad \text{on } (\mathcal{F},\tau_F)$.
As $(E,\mathcal{G})$ is lcscH it is metrisable. By Theorem 2 of Vaughan \cite{Vaughan} there exists an equivalent metric $d$ such that
every bounded set is relatively compact. Further, by second countability there exists a countable and dense subset $E_0 \subseteq E$.
For a general subset $A$ of $E$, $A^0, \overline{A}$ and  $\partial A$ denote the interior, the closure
and the boundary of $A$. Let $B_d(x,r) \equiv B(x,r)$ and $\overline{B}_d(x,r) \equiv \overline{B}(x,r)$ be the open and closed ball with center at $x \in E$ and
radius $r \in \mathbb{R}$. (Observe that e.g. $\overline{B}_d(x,r) = \emptyset$ for $r<0$.) Notice that every closed ball is bounded and therefore is compact.
Given a subset $D \subseteq [0,\infty)$ we introduce the family
\begin{equation} \label{fub}
 \mathcal{U}(D):= \{ \bigcup_{i=1}^m \overline{B}(x_i,r_i): m \in \mathbb{N}, x_i \in E_0, r_i \in D, 1 \le i \le m\}.
\end{equation}
If we agree to say that a closed ball with center in $E_0$ and radius in $D$ is called a \emph{closed $D$-ball}, then
$\mathcal{U}(D)$ is the set of all finite unions of closed $D$-balls. Moreover, let
\begin{equation} \label{Q-continuityset}
 \mathcal{K}_Q:=\{K \in \mathcal{K}: Q(\partial_F \mathcal{H}(K))=0\}=\{K \in \mathcal{K}: Q(\mathcal{H}(K)=Q(\mathcal{H}(K^0)\},
\end{equation}
where $\partial_F \textbf{A}$ gives the boundary of a set $\textbf{A} \subseteq \mathcal{F}$ with respect to the Fell-topology. Here, the
second equality follows from Lemma 4.3 in Ferger \cite{Ferger2}.
Deviating from the usual naming convention, we call $K \in \mathcal{K}_Q$ a $Q$-\emph{continuity set}.\\

Next, we specify the set $D$. For that purpose consider
$$
 R_\pm(x,r):=\{s>0: \overline{B}(x,r\pm s) \text{ is a } Q\text{-continuity set}\}, r>0.
$$
It follows from (3.3) of Ferger \cite{Ferger3} that the sets $\mathcal{H}(\overline{B}(x,r+s), s>0,$ are pairwise disjoint, whence the complement
$R_+(x,r)^c= \{s>0: Q(\partial_F \mathcal{H}(\overline{B}(x,r+s))>0\}$ is denumerable. Similarly, $R_-(x,r)^c$ is denumerable as well.
(Here, notice that $\mathcal{H}(\overline{B}(x,r - s))= \mathcal{H}(\emptyset)=\emptyset$. Therefore $\overline{B}(x,r - s)$ is a $Q$-continuity set for all $s>r$.)
Thus the set
$$
 \bigcup_{x \in E_0, r \in \mathbb{Q}_+} R_+(x,r)^c \cup \bigcup_{x \in E_0, r \in \mathbb{Q}_+} R_-(x,r)^c
$$
is still denumerable and so
$$
 R:= \bigcap_{x \in E_0, r \in \mathbb{Q}_+} R_+(x,r) \cap \bigcap_{x \in E_0, r \in \mathbb{Q}_+} R_-(x,r)
$$
lies dense $[0,\infty)$. (Here $\mathbb{Q}_+$ denotes the set of all positive rational numbers.) In particularly, there exists a sequence $(s_k)_{k \in \mathbb{N}}$ in $R$ such that $s_k \downarrow 0$.
Finally, we define
$$D:=\{r + s_k: r \in \mathbb{Q}_+, k \in \mathbb{N} \} \cup \{r - s_k: r \in \mathbb{Q}_+, k \in \mathbb{N} \}.$$
Notice, that $D=D(E_0,Q)$ depends on $E_0$ and $Q$. It is countable and lies dense in $[0,\infty),$ i.e. $[0,\infty) \subseteq \overline{D}.$
By constuction, every closed $D$-ball is a $Q$-continuity set.\\

\begin{theorem} \label{UD} The following two statements are equivalent:
\begin{itemize}
\item[(1)] $Q_\alpha \longrightarrow_w Q \quad \text{on } (\mathcal{F},\tau_F).$
\item[(2)]
\begin{equation} \label{convergencedeterminingclass}
 \lim_\alpha Q_\alpha(\mathcal{H}(U))=Q(\mathcal{H}(U)) \quad \text{for all }  U \in \mathcal{U}(D).
\end{equation}
\end{itemize}
\end{theorem}

\begin{proof} Assume that (2) holds. Let
$$
 \mathcal{D}:=\{B(x,r): x \in E_0, r \in \mathbb{Q}_+\} \quad \text{and} \quad \overline{\mathcal{D}}:=\{\overline{B}(x,r): x \in E_0, r \in \mathbb{Q}_+\}.
$$
Then the family
\begin{equation} \label{cbase}
 \{ \textbf{B}=\bigcap_{i=1}^m \mathcal{M}(C_i) \cap \bigcap_{j=1}^k \mathcal{H}(D_j): m \in \mathbb{N}, k \in \mathbb{N}_0, C_1,\ldots,C_m \in \overline{\mathcal{D}}, D_1, \ldots, D_k \in \mathcal{D}\}
\end{equation}
is known to be a countable base for $\tau_F$, confer Schneider and Weil \cite{Weil} or Ferger \cite{Ferger0}. Consequently:
\begin{equation} \label{countopen}
 \text{Every open } \textbf{O} \in \tau_F \text{ is a countable union of such base-sets } \textbf{B}.
\end{equation}

We will show that in turn every base-set $\textbf{B}$ can be represented as a countable union as follows:
\begin{equation} \label{cunion}
 \textbf{B}= \bigcup_{\underline{k},\underline{l}} \mathcal{M}(U_{\underline{k}}) \cap \mathcal{H}(B_{l_1}) \cap \ldots \mathcal{H}(B_{l_k}),
\end{equation}
where the union extends over all $\underline{k}=(k_1,\ldots,k_m) \in \mathbb{N}^m$ and all $\underline{l}=(l_1,\ldots,l_k) \in \mathbb{N}^k.$
Moreover, all involved sets $U_{\underline{k}}$ are elements of $\mathcal{U}(D)$ and $B_{l_1},\ldots,B_{l_k} \in \mathcal{U}(D)$ are actually single closed $D$-balls.
For the proof of (\ref{cunion}) consider a general set $\textbf{B}$ from the countable base (\ref{cbase}):
\begin{equation} \label{base}
\textbf{B}=\bigcap_{i=1}^m \mathcal{M}(C_i) \cap \bigcap_{j=1}^k \mathcal{H}(D_j).
\end{equation}
Here, for each $1 \le i \le m$ there are $x_i \in E_0$ and $r_i \in \mathbb{Q}_+$ such that $C_i=\overline{B}(x_i,r_i)$.
Then $C_i^{(k)}:=\overline{B}(x_i,r_i+s_k)$ is a closed $D$-ball for every $ k \in \mathbb{N}$. Furthermore,  $C_i^{(k)} \downarrow C_i, k \rightarrow \infty$.
By Lemma 4.5 in Ferger \cite{Ferger2} we obtain that $\mathcal{H}(C_i)= \bigcap_{k \in \mathbb{N}} \mathcal{H}(C_i^{(k)})$, whence by complementation
it follows that
\begin{equation} \label{miss}
 \mathcal{M}(C_i)= \bigcup_{k \in \mathbb{N}} \mathcal{M}(C_i^{(k)}) \quad \text{for all } 1 \le i \le m.
\end{equation}
For every $1 \le j \le k$ the sets $D_j$ are equal to $B(z_j,t_j)$, where $z_j \in E_0$ and $t_j \in \mathbb{Q}_+$.
Similarly as above $D_j^{(l)}=\overline{B}(z_j,t_j-s_l)$ is a closed $D$-ball for each $ l \in \mathbb{N}$ (possibly empty, namely when $s_l>t_j$.) Clearly, $D_j=\bigcup_{l \in \mathbb{N}} D_j^{(l)}$ and therefore
\begin{equation} \label{hit}
 \mathcal{H}(D_j)= \bigcup_{l \in \mathbb{N}} \mathcal{H}(D_j^{(l)}) \quad \text{ for all } 1 \le j \le k.
\end{equation}
If we substitute (\ref{miss}) and (\ref{hit}) into (\ref{base}), then the distributive law for sets gives us:
\begin{equation} \label{firstunion}
 \textbf{B} = \bigcup_{\underline{k},\underline{l}} \mathcal{M}(C_1^{(k_1)}) \cap \ldots \cap \mathcal{M}(C_m^{(k_m)}) \cap \mathcal{H}(D_1^{(l_1)}) \cap \ldots \cap \mathcal{H}(D_k^{(l_k)}).
\end{equation}
Notice that $\mathcal{M}(C_1^{(k_1)}) \cap \ldots \cap \mathcal{M}(C_m^{(k_m)})=\mathcal{M}(\cup_{i=1}^m C_i^{(k_i)})$.
So, putting $U_{\underline{k}}:=\bigcup_{i=1}^m C_i^{(k_i)}$ and $B_{l_j}:=D_j^{(l_j)}, 1 \le j \le k,$ gives the representation (\ref{cunion}). Moreover, conclude that $U_{\underline{k}} \in \mathcal{U}(D)$
for all $\underline{k} \in \mathbb{N}^m$ and recall that $B_{l_1},\ldots,B_{l_k}$ are closed $D$-balls for all $\underline{l} \in \mathbb{N}^k$ as announced.\\

Introduce the family
$$
 \underline{\mathcal{C}} := \{\mathcal{M}(U) \cap \mathcal{H}(B_1)\cap \mathcal{H}(B_k):U \in \mathcal{U}(D), k \in \mathbb{N}_0, B_1, \ldots, B_k \text{ are closed } D\text{-balls} \}.
$$
It follows from (\ref{countopen}) and (\ref{cunion}) that every open set $\textbf{O} \in \tau_F$ is a countable union
of sets from $\underline{\mathcal{C}}$. One easily verifies that $\underline{\mathcal{C}} \subseteq \underline{\mathcal{B}}_F$ is a $\pi$-system, i.e. closed under finite intersections. Moreover,
\begin{equation} \label{convdetclass}
 Q_\alpha(\textbf{C}) \rightarrow Q(\textbf{C}) \quad \text{for all } \textbf{C} \in \underline{\mathcal{C}}.
\end{equation}
We prove this by induction on $k \in \mathbb{N}_0$. For $k=0$ the sets in $\underline{\mathcal{C}}$ simplify to $\textbf{C}=\mathcal{M}(U)$. Thus the
convergence in (\ref{convdetclass}) holds as a consequence of assumption (2) and the complementation rule for probabilities. Next, consider $k \ge 1$. Check that
\begin{eqnarray*}
 & &\mathcal{M}(U) \cap \mathcal{H}(B_1) \cap \ldots \cap \mathcal{M}(B_k)\\
 &=&\mathcal{M}(U) \cap \mathcal{H}(B_1) \cap \ldots \cap \mathcal{M}(B_{k-1}) \setminus \mathcal{M}(U \cup B_k) \cap \mathcal{H}(B_1) \cap \ldots \cap \mathcal{M}(B_{k-1}).
\end{eqnarray*}
Therefore, (\ref{convdetclass}) follows from the induction hypothesis, because clearly $U \cup B_k \in \mathcal{U}(D)$.

Now, as in the proof of Theorem 2.2 in Billingsley \cite{Bill} we can deduce that $\liminf_\alpha Q_\alpha(\textbf{O}) \ge Q(\textbf{O})$ for all
open $\textbf{O} \in \tau_F$, which as we know is equivalent to (1).\\

Finally, assume that (1) holds. Then
$$
 \lim_\alpha Q_\alpha(\mathcal{H}(K))=Q(\mathcal{H}(K)) \quad \text{for all }  K \in \mathcal{K}_Q
$$
by Theorem 1.7.7 of Molchanov \cite{Molchanov}. But $(\star) \; \mathcal{U}(D) \subseteq \mathcal{K}_Q$, because, as we have seen, every closed ball is compact
and so is each finite union of them. Moreover, for closed $D$-balls $B_1,\ldots,B_m$ it follows that
$$
 0 \le Q(\partial_F \mathcal{H}(\cup_{i=1}^m B_i))= Q(\partial_F \cup_{i=1}^m \mathcal{H}(B_i)) \le Q(\cup_{i=1}^m \partial_F \mathcal{H}(B_i)) \le \sum_{i=1}^m Q(\partial_F \mathcal{H}(B_i))=0,
$$
whence $\bigcup_{i=1}^m B_i \in \mathcal{K}_Q$ and thus $(\star)$ holds.
This shows the validity of (2).
\end{proof}

If $(P_\alpha)_{\alpha \in \mathbb{N}}$ is a sequence, then there are comparable results in the literature, where $\mathcal{U}(D)$ is replaced
by another family, say $\mathcal{V}$. In Kallenberg \cite{Kall} and Norberg \cite{Norberg} $\mathcal{V}$  is a \emph{separating class}, confer Molchanov \cite{Molchanov} for its definition. In Salinetti
and Wets \cite{SalWe} $E$ is a finite dimensional linear space and $\mathcal{V}=\mathcal{U}(\mathbb{Q}_+) \cap \mathcal{K}_Q$ and in Pflug \cite{Pflug} $E=\mathbb{R}^d$ and $\mathcal{V}$ is equal to the family of all finite unions of compact rectangles with rational endpoints. Here, as it is the case in Salinetti and Wets \cite{SalWe} the unions (and not the single rectangles) must be $Q$-continuity sets.
A major advantage of our ''convergence determining class'' $\mathcal{V}=\mathcal{U}(D)$ is that it is tailor-made for describing epi-convergence in distribution. We will see that this is so, because in our result the single closed balls are $Q$-continuity sets.\\

A reformulation of Theorem \ref{UD} in terms of random closed sets as in (\ref{dconvrcs}) reads as follows:
$C_\alpha \stackrel{\mathcal{D}}{\longrightarrow} C \; \text{in } (\mathcal{F},\tau_F)$ if and only if
$$
 \lim_\alpha \mathbb{P}_\alpha(C_\alpha \cap U \neq \emptyset) = \mathbb{P}(C \cap U \neq \emptyset) \; \text{for all } U \in \mathcal{U}(D)
$$
with $\mathbb{P}(C \cap U \neq \emptyset)= \mathbb{P}(C \cap U^0 \neq \emptyset).$
Here, we use the description of $Q$-continuity sets given in (\ref{Q-continuityset}).

\section{Weak convergence of probability measures on $(S,\tau_e)$}
We begin with several rather simple necessary conditions for weak convergence, which will be of good use later on.\\

\begin{lemma} If $P_\alpha \rightarrow_w P$ in $(S,\tau_e)$, then the following statements hold:
\begin{align}
\liminf_\alpha P_\alpha (\cap_{i=1}^m \{I_{G_j} < a_j\})
    &\ge P(\cap_{i=1}^m \{I_{G_j} < a_j\}), \label{liminfG}\\
\liminf_\alpha P_\alpha(\cap_{i=1}^m \{I_{K_j} > a_j\})
    &\ge P(\cap_{i=1}^m \{I_{K_j} > a_j\}), \label{liminfK} \\
\limsup_\alpha P_\alpha(\cap_{i=1}^m \{I_{K_j} \le a_j\})
    &\ge P(\cap_{i=1}^m \{I_{K_j} \le a_j\}), \label{limsupK}\\
\limsup_\alpha P_\alpha(\cap_{i=1}^m \{I_{G_j} \ge a_j\})
    &\ge P(\cap_{i=1}^m \{I_{G_j} \ge a_j\}) \label{limsupG}
\end{align}
for all $m \in \mathbb{N}, G_1,\ldots,G_m \in \mathcal{G}, K_1,\ldots,K_m \in \mathcal{K}$ and $a_1,\ldots,a_m \in \mathbb{R}.$
\end{lemma}

\begin{proof} Since $\bigcap_{i=1}^m \{I_{G_j} < a_j\} \in \tau_e$, and $\bigcap_{i=1}^m \{I_{K_j} > a_j\} \in \tau_e$,  the first two assertions (\ref{liminfG}) and (\ref{liminfK}) follow from the Portmanteau-Theorem.
Similarly, as $\bigcap_{i=1}^m \{I_{K_j} \le a_j\}$ and  $\bigcap_{i=1}^m \{I_{G_j} \ge a_j\}$ both are $\tau_e$-closed another application of the Portmanteau-Theorem yields (\ref{limsupK}) and (\ref{limsupG})
\end{proof}

A combination of (\ref{liminfG})-(\ref{limsupG}) leads to a result involving usual limits.\\

\begin{proposition} \label{necessary} Weak convergence $P_\alpha \rightarrow_w P$ in $(S,\tau_e)$ entails
\begin{align}
 \lim_\alpha P_\alpha(\cap_{i=1}^m \{I_{K_j} \le a_j\}) = P(\cap_{i=1}^m \{I_{K_j} \le a_j\}), \label{limlessthanorequal} \\
 \lim_\alpha P_\alpha(\cap_{i=1}^m \{I_{K_j} > a_j\}) = P(\cap_{i=1}^m \{I_{K_j} > a_j\}). \label{limgreater}
\end{align}
Here, (\ref{limlessthanorequal}) and (\ref{limgreater}), respectively, hold for all $m \in \mathbb{N}, a_1,\ldots,a_m \in \mathbb{R}$ and
$K_1,\ldots,K_m \in \mathcal{K}$ satisfying the continuity condition
$$
  P(I_{K_j} \le a_j)=P(I_{K_j^0} < a_j), j=1,\ldots,m.
$$
\end{proposition}

\begin{proof} As to the proof of (\ref{limlessthanorequal}) observe that
\begin{eqnarray*}
P(\cap_{i=1}^m \{I_{K_j^0} < a_j\}) &\le& \liminf_\alpha P_\alpha(\cap_{i=1}^m \{I_{K_j^0} < a_j\}) \quad \text{by }(\ref{liminfG})\\
                                    &\le& \liminf_\alpha P_\alpha(\cap_{i=1}^m \{I_{K_j} \le a_j\}) \quad \text{since} I_{K_j^0} \ge I_{K_j} \text{ and } (-\infty,a_j) \subseteq (-\infty,a_j]\\
                                    &\le& \limsup_\alpha P_\alpha(\cap_{i=1}^m \{I_{K_j} \le a_j\})\\
                                    &\le& P(\cap_{i=1}^m \{I_{K_j} \le a_j\}) \hspace{1.7cm} \text{by }(\ref{limsupK})\\
                                    &=&   P(\cap_{i=1}^m \{I_{K_j^0} < a_j\}) \hspace{1.7cm} \text{see below.}
\end{eqnarray*}
So, the assertion (\ref{limlessthanorequal}) follows, once we have shown the last equation. For this purpose, let $M_j:=\{I_{K_j^0} < a_j\}$ and $N_j:=\{I_{K_j} \le a_j\}, j=1,\ldots m$. Then $M_j \subseteq N_j$ for all $j$, because $I_{K_j^0} \ge I_{K_j}$ and $(-\infty,a_j) \subseteq (\infty,a_j]$. As a consequence, $M:=\bigcap_{j=1}^m M_j \subseteq \bigcap_{j=1}^m N_j =: N$ and thus:
\begin{eqnarray*}
0 &\le& Q(N)-Q(M)=  Q(N \setminus M)=Q(N \cap (\cup_{j=1}^m M_j^c))= Q(\cup_{j=1}^m(N \cap M_j^c))\\
  &\le& Q(\cup_{j=1}^m(N_j \cap M_j^c)) \le \sum_{j=1}^m Q(N_j \setminus M_j)= \sum_{j=1}^m Q(N_j)-Q(M_j)=0,
\end{eqnarray*}
where the last equality follows from our assumption. Consequently, $Q(N)=Q(M)$ as desired.
Analogously, by using (\ref{liminfK}) and (\ref{limsupG}) one proves (\ref{limgreater}).
\end{proof}

Proposition \ref{necessary} gives two necessary conditions for weak convergence. Our next result yields two necessary and sufficient conditions.\\

\begin{theorem} \label{sufficient} Let $D:=D(E_0 \times \mathbb{Q},P \circ \phi^{-1})$. Then the following statements (1)-(3) are equivalent:
\begin{itemize}
\item[(1)] $P_\alpha \rightarrow_w P$ in $(S,\tau_e)$
\item[(2)] $\lim_\alpha P_\alpha(\cap_{j=1}^m \{I_{\overline{B}(x_j,r_j)} \le r_j+\alpha_j\}) = P(\cap_{j=1}^m \{I_{(\overline{B}(x_j,r_j))^0} \le r_j+ \alpha_j\}).$
\item[(3)] $\lim_\alpha P_\alpha(\cap_{j=1}^m \{I_{\overline{B}(x_j,r_j)} > r_j+\alpha_j\}) = P(\cap_{j=1}^m \{I_{(\overline{B}(x_j,r_j))^0} > r_j+ \alpha_j\}).$
\end{itemize}
Here, the equalities in (2) and (3), respectively, hold for all
$m \in \mathbb{N}, x_1,\ldots,x_m \in E_0, r_1,\ldots,r_m \in D, \alpha_1,\ldots,\alpha_m \in \mathbb{Q}.$ (Notice that the closed balls
$\overline{B}(x_j,r_j)$ satisfy the continuity condition (\ref{continuitycondition}) below.)
\end{theorem}

\begin{proof} The necessity of (2) or (3), respectively, for weak convergence (1) follows from Proposition \ref{necessary}, because closed balls are compact and $r_j \in D$ means that $\overline{B}_{d \times u}((x_j,\alpha_j),r_j)$ is a $P \circ \phi^{-1}$-continuity set, which by Lemma \ref{continuityset} in the appendix is equivalent to
\begin{equation} \label{continuitycondition}
P(I_{\overline{B}(x_j,r_j)} \le r_j+\alpha_j )= P(I_{(\overline{B}(x_j,r_j))^0} < r_j+\alpha_j).
\end{equation}
As to sufficiency we will see that it is enough to show that
\begin{equation} \label{Pphiweak}
P_\alpha \circ \phi^{-1} \rightarrow_w P \circ \phi^{-1} \quad \text{on } (\mathcal{F}(E \times \mathbb{R}),\tau_F(E \times \mathbb{R})).
\end{equation}
So, we are dealing with $(E \times \mathbb{R}, d \times u)$ as carrier space instead of $(E,d)$. (The product-metric is specified in the Appendix below.) Assume that (2) holds. Observe that
$$
 \mathcal{U}(D) = \{\bigcup_{i=1}^m \overline{B}_{d \times u}((x_i,\alpha_i),r_i): m \in \mathbb{N}, (x_i,\alpha_i) \in E_0 \times \mathbb{Q}, r_i \in D, i=1,\ldots,m\}.
$$
Let $U=\bigcup_{i=1}^m B_i \in \mathcal{U}(D)$, i.e. each $B_i$ is equal to the closed ball $\overline{B}_{d \times u}((x_i,\alpha_i),r_i), i=1,\ldots,m.$ Put $Q_\alpha := P_\alpha \circ \phi^{-1}$ and $Q := P \circ \phi^{-1}$. The \emph{inclusion-exclusion formula} yields:
\begin{equation} \label{ief}
Q_\alpha(\mathcal{H}(U))=Q_\alpha(\cup_{i=1}^m(\mathcal{H}(B_i)))=\sum_{k=1}^m (-1)^{k+1} \sum_{1 \le i_1<\ldots<i_k\le m} Q_\alpha(\cap_{s=1}^k \mathcal{H}(B_{i_s})).
\end{equation}
For each summand on the right side of equation (\ref{ief}) it follows that:
\begin{eqnarray*}
Q_\alpha(\cap_{s=1}^k \mathcal{H}(B_{i_s}))&=&P_\alpha(\cap_{s=1}^k \{f \in S: \text{epi}(f) \in \mathcal{H}(B_{i_s})\})\\
&=&P_\alpha(\cap_{s=1}^k \{f \in S: \text{epi}(f) \cap B_{i_s} \neq \emptyset\})\\
&=&P_\alpha(\cap_{s=1}^k \{I_{\overline{B}_d(x_{i_s},r_{i_s})} \le r_{i_s}+\alpha_{i_s}\}) \quad \text{by } (\ref{closedballisclosedrectangle}) \text{ and } (\ref{epi1}).
\end{eqnarray*}
Consequently by (2), $$Q_\alpha(\cap_{s=1}^k \mathcal{H}(B_{i_s})) \rightarrow P(\cap_{s=1}^k \{I_{\overline{B}_d(x_{i_s},r_{i_s})} \le r_{i_s}+\alpha_{i_s}\})=Q(\cap_{s=1}^k \mathcal{H}(B_{i_s}))$$
upon noticing (\ref{closedballisclosedrectangle}) and (\ref{epi1}) again.
With (\ref{ief}) we obtain that $Q_\alpha(\mathcal{H}(U)) \rightarrow Q(\mathcal{H}(U))$ for all $U \in \mathcal{U}(D).$ Thus the weak convergence (\ref{Pphiweak}) follows from Theorem \ref{UD}. Deduce from (\ref{Pphiweak}) that
$$
 P_\alpha \circ \phi^{-1} \rightarrow_w P \circ \phi^{-1} \quad \text{on the subspace } (\mathcal{E},\sigma).
$$
This yields $P_\alpha \rightarrow_w P$ by the Continuous Mapping Theorem, because
$$P_\alpha=(P_\alpha \circ \phi^{-1})\circ (\phi^{-1})^{-1}$$
and $\phi^{-1}:(\mathcal{E},\sigma) \rightarrow (S,\tau_e)$
is continuous.

Finally, assume that (3) holds. Then
\begin{eqnarray*}
Q_\alpha(\mathcal{M}(U))&=& Q_\alpha(\mathcal{M}(\cup_{i=1}^m B_i))= Q_\alpha(\cap_{i=1}^m \mathcal{M}(B_i))\\
                        &=&P_\alpha(\cap_{i=1}^m \{f \in S: \text{epi}(f) \in \mathcal{M}(B_i)\})\\
                        &=&P_\alpha(\cap_{i=1}^m \{f \in S: \text{epi}(f) \cap B_i = \emptyset\})\\
                        &=&P_\alpha(\cap_{i=1}^m \{I_{\overline{B}_d(x_i,r_i)} > r_i+\alpha_i\}) \quad \text{by } (\ref{closedballisclosedrectangle}) \text{ and } (\ref{epi1}).
\end{eqnarray*}
So, weak convergence (\ref{Pphiweak}) follows from Theorem \ref{UD} by complementation, which as shown above results in (1).
\end{proof}

The reformulation of our results in terms of normal integrands is obvious. For instance, a net $(Z_\alpha)$ of normal integrands epi-converges in distribution to a normal integrand $Z$ if and only if
$$
 \mathbb{P}_\alpha(\inf_{t \in \overline{B}_d(x_i,r_i)}Z_\alpha(t)>r_i+\alpha_i, i=1,\ldots,m) \rightarrow \mathbb{P}(\inf_{t \in \overline{B}_d(x_i,r_i)}Z(t)>r_i+\alpha_i, i=1,\ldots,m)
$$
for all $m \in \mathbb{N}, x_1,\ldots,x_m \in E_0, r_1,\ldots,r_m \in D, \alpha_1,\ldots,\alpha_m \in \mathbb{Q}.$ In this form we can immediately compare it with the equivalent characterisation of Molchanov's \cite{Molchanov} Proposition 5.3.20:
$$
 \mathbb{P}_\alpha(\inf_{t \in K_i}Z_\alpha(t)> t_i, i=1,\ldots,m) \rightarrow \mathbb{P}(\inf_{t \in K_i}Z(t)> t_i, i=1,\ldots,m)
$$
for all $m \in \mathbb{N}, t_1,\ldots,t_m \in \mathbb{R}$ and $K_1,\ldots,K_m$ belonging to a separating class of subsets of $E$ satisfying the condition
$$
P(I_{K_i} \le t_i )= P(I_{K_i^0} < t_i).
$$
For instance the family $\mathcal{K}$ is separating or the family of all finite unions of closed balls with center in $E_0$ and positive rational
radii. In both cases our countable class $\{\overline{B}_d(x,r): x \in E_0, r \in D\}$ is significantly smaller. Similarly, the countable set $\{r+\alpha: r \in D, \alpha \in \mathbb{Q}\}$ is a subset of the real line $\mathbb{R}$. Finally, Molchanov only considers sequences $(Z_n)_{n \in \mathbb{N}}$ and not more generally nets $(Z_\alpha)_{\alpha \in A}$, as we do. In the case of sequences and $E=\mathbb{R}^d$ there are further characterisations for epi-convergence in distribution. In their Theorem 3.14 Salinetti and Wets \cite{SalWe} show that if $(Z_n)$ is almost surely equi-lower semicontinuous, then epi-convergence in distribution is equivalent to convergence of the finite-dimensional distributions (fidis). Gersch \cite{Gersch}, Theorems 2.19 and 2.25, requires merely \emph{stochastically} equi-lower semicontinuity and gives a sufficient condition, which again involve the marginals of the $Z_n$, but in a more complicated way in comparison to the convergence of the fidis. However, if in addition the limit process $Z$ is stochastically uniformly lower semicontinuous, then epi-convergence in distribution and convergence of the fidis are equivalent. This equivalence also holds for convex $Z$ and $Z_n, n \in \mathbb{N}$ as Ferger \cite{Ferger1} shows in Theorem 2.12 and Proposition 2.13. In statistical applications the Skorokhod-space $(D(\mathbb{R}^d),s)$ plays an important role, because many empirical processes have trajectories lying in that function space. From Proposition 2.1 of Ferger \cite{Ferger4} it follows that, if $Z_n \stackrel{\mathcal{D}}{\rightarrow} Z$ in $(D(\mathbb{R}^d),s)$, then $\overline{Z}_n \stackrel{\mathcal{D}}{\rightarrow} \overline{Z}$ in $(S(\mathbb{R}^d),e)$. Here, $s$ denotes the Skorokhod-metric
and $\overline{f}$ is the lsc regularization of a function $f \in D(\mathbb{R}^d)$.\\

\begin{remark} \label{epiindistribution} Since $\phi: (S,\tau_e) \rightarrow (\mathcal{E},\sigma)$ is a homeomorphism, the Continuous Mapping Theorem yields that a net $(Z_\alpha)$ of normal integrands epi-converges in distribution to a normal
integrand $Z$, $Z_\alpha \stackrel{\mathcal{D}}{\rightarrow} Z$ in $(S,\tau_e)$, if and only if the pertaining epi$(Z_\alpha)$ converge in distribution to epi$(Z)$ in the space $E \times \mathbb{R}$, i.e.
$\text{epi}(Z_\alpha) \stackrel{\mathcal{D}}{\rightarrow} \text{epi}(Z)$ in $(\mathcal{F}(E \times \mathbb{R}), \tau_F(E \times \mathbb{R}))$.
In the literature so far this was taken as the definition of epi-convergence in distribution.
\end{remark}

\vspace{0.2cm}
\begin{remark} \label{interiorofclosedball} An important point in our investigations is that Vaughan's metric ensures that every closed bounded set is
compact. This is fulfilled if $E$ is a finite-dimensional normed linear space.
Moreover, notice that in normed linear spaces $(\overline{B}(x,r))^0=B(x,r)$. This need not be true more generally in metric spaces even if they are lcscH. As an example consider
a countable set $E$ with at least two elements endowed with the discrete metric. This space is lcscH, but $(\overline{B}(x,1))^0= E^0 =E \nsubseteq \{x\}= B(x,1).$ If $E$ is finite,
then $(E,d)$ is even compact and thus in particularly $d$ has the property of Vaughan's metric as actually every subset is compact.
\end{remark}

\section{Applications to sets of $\epsilon$-optimal solutions of normal integrands}
For $f \in S(E)$ and $\epsilon \ge 0$ let
$$
 A(f,\epsilon):=\{t \in E: f(t)\le I_E(f)+\epsilon\}
$$
be the set of all \emph{$\epsilon$-optimal solutions of $f$}. This set is closed, because $A(f,\epsilon)= \{f \le \alpha\}$ with $\alpha:= I_E(f)+\epsilon \in [-\infty,\infty]$. Now, if $I_E(f) \in \mathbb{R}$, then $\alpha \in \mathbb{R}$ and $A(f,\epsilon)=\{f \le \alpha\} \in \mathcal{F}$ by lower semicontinuity of $f$. If $I_E(f)= \infty$, then $A(f,\epsilon)=\{f \le \infty\}=E \in \mathcal{F}.$
Finally, in case that $I_E(f)=-\infty$ then $A(f,\epsilon)=\{f= -\infty\}= \bigcap_{n \in \mathbb{N}} \{f \le -n\} \in \mathcal{F}$ again by lower semicontinuity of $f$ and since $\mathcal{F}$ is closed under intersection.
So, the assignment $(f,\epsilon) \mapsto A(f,\epsilon)$ defines a map
\begin{equation} \label{Amap}
A:S \times \mathbb{R}_+ \rightarrow \mathcal{F}.
\end{equation}

For $\epsilon=0$ one obtains $A(f,0)=\{t \in E: f(t)=I_E(f)\}= $Argmin$(f)$, the set of all minimizing points of the function $f$.

\par
A very nice property of epi-convergence is formulated in our next result.\\

\begin{proposition} \label{PKconvergence} Let  $(f_n)_{n \in \mathbb{N}}$ be a sequence of lsc
functions, which epi-converges to some $f \in S$, i.e. $f_n \rightarrow f$ in $(S,\tau_e)$. Moreover, assume $(\epsilon_n)_{n \in \mathbb{N}} \ge 0 $ is a sequence
of non-negative numbers such that $\limsup_{n \rightarrow \infty} \epsilon_n \le \epsilon$.
Then
\begin{equation} \label{PKconv}
 \text{PK-}\limsup_{n \rightarrow} A(f_n,\epsilon_n) \subseteq A(f,\epsilon)).
\end{equation}
Here, $\text{PK-}\limsup$ denotes the \emph{Painlev\'{e}-Kuratowski outer limit} of a sequence of sets.
\end{proposition}

\begin{proof} Let $t \in \text{PK-}\limsup_{n \rightarrow} A(f_n,\epsilon_n)$. By definition of the outer limit there exists a subsequence $(n_j)_{j \in \mathbb{N}}$ of the natural numbers and points $t_{n_j} \in
A(f_{n_j},\epsilon_{n_j})$ with $t_{n_j} \rightarrow t$ as $j \rightarrow \infty.$ Assume that $t \notin A(f,\epsilon)$. Then there exists some $s \in E$ such that $f(t)>f(s)+\epsilon$, because otherwise
$f(t) \le f(s)+\epsilon$ for all $s \in E$, whence $I_E(f) \ge f(t)-\epsilon$ and thus $f(t) \le I_E(f)+\epsilon$ in contradiction to $t \notin A(f,\epsilon)$. It follows from the characteristion of epi-convergence, confer Theorem 5.3.2 (ii) in Molchanov \cite{Molchanov}, that
there exists some sequence $(s_n)$ such that $s_n \rightarrow s$ and $f_n(s_n) \rightarrow f(s)$. Since $t_{n_j} \in A(f_{n_j},\epsilon_{n_j})$, we have that
$f_{n_j}(t_{n_j}) \le I_E(f_{n_j})+\epsilon_{n_j} \le f_{n_j}(s_{n_j})+\epsilon_{n_j}$ and so
\begin{equation} \label{snj}
 f_{n_j}(s_{n_j}) \ge f_{n_j}(t_{n_j}) - \epsilon_{n_j} \quad \forall \; j \in \mathbb{N}.
\end{equation}
Conclude that
\begin{eqnarray*}
f(t) &>& f(s) + \epsilon \\
     &=& \lim_{n \rightarrow \infty} f_n(s_n) + \epsilon\\
     &=& \liminf_{j \rightarrow \infty} f_{n_j}(s_{n_j}) + \epsilon\\
     &\ge& \liminf_{j \rightarrow \infty} (f_{n_j}(t_{n_j})- \epsilon_{n_j}) + \epsilon \hspace{3cm} \text{by } (\ref{snj})\\
     &\ge& \liminf_{j \rightarrow \infty} f_{n_j}(t_{n_j}) + \liminf_{j \rightarrow \infty} (-\epsilon_{n_j}) + \epsilon\\
     &=&  \liminf_{j \rightarrow \infty} f_{n_j}(t_{n_j}) - \limsup_{j \rightarrow \infty} \epsilon_{n_j} + \epsilon\\
     &\ge& \liminf_{j \rightarrow \infty} f_{n_j}(t_{n_j}) - \limsup_{n \rightarrow \infty} \epsilon_n + \epsilon\\
     &\ge& \liminf_{j \rightarrow \infty} f_{n_j}(t_{n_j}) \hspace{4.7cm} \text{by assumption on } (\epsilon_n)\\
     &\ge& f(t),
\end{eqnarray*}
where the last inequality follows from Theorem 5.3.2 (ii) in Molchanov \cite{Molchanov} since $(f_{n_j})_{j \in \mathbb{N}}$ as a subsequence epi-converges to $f$ as well.
So, we arrive at a contradiction, which finishes our proof.
\end{proof}

Let $\mathcal{T}:=\{[0,r): 0< r \in \mathbb{R}\} \cup \{\emptyset,\mathbb{R}_+\}$ be the \emph{left-order topology} on $\mathbb{R}_+:=[0,\infty)$. Then the assumption $\limsup_{n \rightarrow \infty} \epsilon_n \le \epsilon$ is
equivalent to $\epsilon_n \rightarrow \epsilon$ in $(\mathbb{R}_+,\mathcal{T})$. If $\epsilon=0$, then $\epsilon_n \rightarrow 0$ in the natural topology $\mathcal{T}_n$ on $\mathbb{R}_+$ since all $\epsilon_n$ are non-negative and therefore
$\liminf_{n \rightarrow \infty} \epsilon_n \ge 0$.
For this special case we obtain Proposition 2.9 of Attouch \cite{Attouch}, who however considers more generally arbitrary topological spaces $(E,\mathcal{G})$.\\

Next, let $\tau_{uF}$ be the \emph{upper Fell topology}, which is generated by the family $\mathcal{S}_{uF}:=\{\mathcal{M}(K): K \in \mathcal{K}\}$. Since $\mathcal{S}_{uF} \subseteq \mathcal{S}$, the upper Fell-topology is weaker
than the Fell-topology. In Lemma 2.2, Vogel \cite{Vogel} shows in case $E=\mathbb{R}^d$ that a sequence $(F_n)_{n \in \mathbb{N}} \subseteq \mathcal{F}$ of closed sets in $E$ converges to some $F \in \mathcal{F}$ in the
upper Fell-topology, i.e.

\begin{equation} \label{PKMiss}
 F_n \rightarrow F \text{ in } (\mathcal{F},\tau_{uF}) \; \text{ if and only if } \; \text{PK-}\limsup_{n \rightarrow \infty} F_n \subseteq F.
\end{equation}

This remains valid more generally for $(E,\mathcal{G})$ lcscH, confer Proposition 2.18 (b) in Ferger \cite{Ferger0}. Our next result yields continuity of the map $$A:S \times \mathbb{R}_+ \rightarrow \mathcal{F}.$$\\

\begin{corollary} \label{Aiscontinuous} Let $\tau_e \times \mathcal{T}$ be the product-topology on $S \times \mathbb{R}_+$. Then
\begin{equation} \label{leftordertopology}
A:(S \times \mathbb{R}_+,\tau_e \times \mathcal{T}) \rightarrow (\mathcal{F},\tau_{uF}) \quad \text{is continuous}.
\end{equation}
In particularly,
\begin{equation} \label{naturaltopology}
A:(S \times \mathbb{R}_+,\tau_e \times \mathcal{T}_n) \rightarrow (\mathcal{F},\tau_{uF}) \quad \text{is continuous}.
\end{equation}
\end{corollary}

\begin{proof} By Corollary 2.79 of Attouch \cite{Attouch} the space $(S,\tau_e)$ is second-countable, which also holds for $(\mathbb{R}_+,\mathcal{T})$ as one easily verifies. Thus the product
$(S \times \mathbb{R}_+,\tau_e \times \mathcal{T})$ is second-countable as well and in particularly, it is first-countable. By Theorem 7.1.3 in Singh \cite{Singh} it suffices to show that
$A$ is sequentially-continuous. So, assume that $(f_n,\epsilon_n) \rightarrow (f,\epsilon)$ in $(S \times \mathbb{R}_+,\tau_e \times \mathcal{T})$. This is equivalent to
$f_n \rightarrow f$ in $(S,\tau_e)$ and $\epsilon_n \rightarrow \epsilon$ in $(\mathbb{R}_+,\mathcal{T})$. But the latter in turn means that $\limsup_{n \rightarrow \infty} \epsilon_n \le \epsilon$.
Therefore, the assertion in (\ref{leftordertopology}) follows from Proposition \ref{PKconvergence} in combination with the equivalence (\ref{PKMiss}). Since $\tau_e \times \mathcal{T}_n \supseteq \tau_e \times \mathcal{T}$,
the second assertion follows from the first one.
\end{proof}

Let $Z$ be a normal integrand and $\epsilon$ a random variable with values in $\mathbb{R}_+$ both defined on some measurable space $(\Omega,\mathcal{A})$. A first useful application of Corollary \ref{Aiscontinuous} yields
measurability of the random set $A(Z,\epsilon)$. In the proof below we will use the following notation: Given a topological space $(X,\mathcal{O})$ the pertaining Borel-$\sigma$ algebra $\sigma(\mathcal{O})$ is denoted by $\mathcal{B}(X)$.\\

\begin{corollary} \label{Aisrandomclosedset} If $Z$ and $\epsilon$ are as above, then $A(Z,\epsilon)$ is a random closed set.
\end{corollary}

\begin{proof} It follows from (\ref{naturaltopology}) that $A$ is $\mathcal{B}(S \times \mathbb{R}_+)-\underline{\mathcal{B}}_{uF}$ measurable, where $\underline{\mathcal{B}}_{uF} := \sigma(\tau_{uF})$. But $\underline{\mathcal{B}}_{uF}=\underline{\mathcal{B}}_{F}$,
because $\underline{\mathcal{B}}_{F} = \sigma(\mathcal{S}_{uF})$ by Lemma 2.1.1 in Schneider and Weil \cite{Weil} and $\mathcal{S}_{uF} \subseteq \tau_{uF} \subseteq \tau_F.$
Since $(S \times \mathbb{R}_+,\tau_e \times \mathcal{T}_n)$ is second-countable, it follows that the Borel-$\sigma$ algebra $\mathcal{B}(S \times \mathbb{R}_+) = \mathcal{B}(S) \otimes \mathbb{B}(\mathbb{R}_+)$.
Infer that $A$ is $\mathcal{B}(S) \otimes \mathbb{B}(\mathbb{R}_+)-\underline{\mathcal{B}}_{F}$ measurable. Deduce from our assumption that the product map $(Z,\epsilon):(\Omega,\mathcal{A}) \rightarrow (S \times \mathbb{R}_+,\mathcal{B}(S) \otimes \mathbb{B}(\mathbb{R}_+))$ is measurable, whence the assertion follows upon noticing that $A(Z,\epsilon)=A \circ (Z,\epsilon)$ is a composition of measurable maps.
\end{proof}

A further utility of Corollary \ref{Aiscontinuous} is that it enables us to apply the Continuous Mapping Theorem (CMT) for random variables in topological spaces, confer Proposition 8.4.16 in G\"{a}nssler and Stute \cite{Stute}.
The random variables $Z, \epsilon$ and $Z_\alpha, \epsilon_\alpha$ occuring in our results below are defined on $(\Omega,\mathcal{A},\mathbb{P})$ and $(\Omega_\alpha,\mathcal{A}_\alpha,\mathbb{P}_\alpha)$, respectively.\\

\begin{theorem} \label{epidistA} Let $(Z_\alpha)$ and $(\epsilon_\alpha)$ be nets of normal integrands and non-negative random variables, respectively.
Assume that
\begin{equation} \label{Zeps}
 (Z_\alpha,\epsilon_\alpha) \stackrel{\mathcal{D}}{\rightarrow} (Z,\epsilon) \; \text{ in } (S \times \mathbb{R}_+,\tau_e \times \mathcal{T}).
\end{equation}
Then $A(Z_\alpha,\epsilon_\alpha) \stackrel{\mathcal{D}}{\rightarrow} A(Z,\epsilon)$ in $(\mathcal{F}, \tau_{uF}).$
This is the same as
\begin{eqnarray} \label{K}
& &\limsup_\alpha \mathbb{P}_\alpha(\bigcap_{K \in \mathcal{K}^*} \{\omega \in \Omega_\alpha: A(Z_\alpha(\omega),\epsilon_\alpha(\omega)) \cap K \neq \emptyset\} ) \nonumber\\
 &\le& \mathbb{P}(\bigcap_{K \in \mathcal{K}^*} \{\omega \in \Omega: A(Z(\omega),\epsilon(\omega)) \cap K \neq \emptyset\} )
\end{eqnarray}
for every collection $\mathcal{K}^* \subseteq \mathcal{K}$ of compact sets in $E$.
\end{theorem}

\begin{proof} Assumption (\ref{Zeps}) and the CMT applied to $A$ in Corollary \ref{Aiscontinuous} gives the first assertion of the theorem. The second one follows from Proposition 2.1 of
Ferger \cite{Ferger2}.
\end{proof}

A legitimate question is what are sufficient conditions for the validity of (\ref{Zeps})? We provide a few answers in:\\

\begin{remark} \label{suffcond} (1) Assume that there is a component-wise convergence, i.e.
\begin{equation} \label{componentwise}
Z_\alpha \stackrel{\mathcal{D}}{\rightarrow} Z \text{ in }  (S,\tau_e) \text{ and }
\epsilon_\alpha \stackrel{\mathcal{D}}{\rightarrow} \epsilon \text{ in } (\mathbb{R}_+,\mathcal{T}_n).
\end{equation}
If $\epsilon$ is almost surely constant, then
\begin{equation} \label{Zepsnatural}
 (Z_\alpha,\epsilon_\alpha) \stackrel{\mathcal{D}}{\rightarrow} (Z,\epsilon) \; \text{ in } (S \times \mathbb{R}_+,\tau_e \times \mathcal{T}_n).
\end{equation}
This follows from Slutsky's Theorem, confer Proposition 8.6.4 in G\"{a}nssler and Stute \cite{Stute}. Now, (\ref{Zepsnatural}) implies (\ref{Zeps}), because $\tau_e \times \mathcal{T}_n \supseteq \tau_e \times \mathcal{T}$.\\

(2) Suppose  that $Z_\alpha$ and $\epsilon_\alpha$ are $\mathbb{P}_\alpha$-independent for each $\alpha$. If component-wise convergence (\ref{componentwise}) holds, where $Z$ and $\epsilon$ are $\mathbb{P}$-independent,
then again (\ref{Zepsnatural}) holds by Theorem 2.8 in Billingsley \cite{Bill}. As we know this is enough for (\ref{Zeps}). A special case for this is when the net $(\epsilon_\alpha)$ is deterministic and convergent with (deterministic) limit $\epsilon$.\\

(3) It should be mentioned that Slutsky's Theorem and Theorem 2.8 in Billingsley \cite{Bill} are formulated only for sequences. However, by using Proposition 8.4.9 in G\"{a}nssler and Stute \cite{Stute} one can see that their proofs can easily be transferred to nets .
\end{remark}

\vspace{0.5cm}
The question arises as to the requirements under which Fell-convergence in distribution is obtained.
The answer involves the family $\textbf{F}_{0,1}:=\{\emptyset\} \cup \{\{x\}: x \in E\}$ of sets with at most one element.\\

\begin{theorem} \label{Fellconvergencein distribution} Assume that (\ref{Zeps}) holds with $\epsilon=0$.
If for every $\eta>0$ there exists a compact $K \subseteq E$ such that
$$
 \liminf_\alpha \mathbb{P}_\alpha(\{\omega \in \Omega_\alpha: \emptyset \neq A(Z_\alpha(\omega),\epsilon_\alpha(\omega)) \subseteq K\}) \ge 1-\eta
$$
and if $\mathbb{P}(\{\omega \in \Omega: Z(\omega) \in \textbf{F}_{0,1}\})=1$, i.e. $Z$ has at most one minimizing point almost surely, then $$A(Z_\alpha,\epsilon_\alpha) \stackrel{\mathcal{D}}{\rightarrow} \text{Argmin}(Z) \text{ in } (\mathcal{F}, \tau_{F}).$$
\end{theorem}

\begin{proof} First, notice that by Lemma 4.6 of Ferger \cite{Ferger2} the set $\textbf{F}_{0,1}$ is closed in $(\mathcal{F},\tau_F)$ and consequently $\{Z \in \textbf{F}_{0,1}\} \in \mathcal{A}$, the domain of $\mathbb{P}.$ It follows from Theorem \ref{epidistA} that $A(Z_\alpha,\epsilon_\alpha) \stackrel{\mathcal{D}}{\rightarrow} \text{Argmin}(Z)$ in $(\mathcal{F}, \tau_{uF}).$
Thus, an application of Theorem 2.9 of Ferger \cite{Ferger2} yields the assertion.
\end{proof}

For Argmin-sets the assumption (\ref{Zeps}) simplifies significantly, because here $\epsilon_\alpha =0$ for all $\alpha \in A$, which is
a special case of the special case in Remark \ref{suffcond} (2).\\

\begin{corollary} \label{Argminsets} Suppose that $Z_n \stackrel{\mathcal{D}}{\rightarrow} Z$ in $(S,\tau_e)$, where $Z$ has at most one minimizing point almost surely. If
for every $\eta>0$ there exists a compact $K \subseteq E$ such that
$$
 \liminf_\alpha \mathbb{P}_\alpha(\{\omega \in \Omega_\alpha: \emptyset \neq \text{Argmin}(Z_\alpha(\omega)) \subseteq K\}) \ge 1-\eta,
$$
then
$$\text{Argmin}(Z_\alpha) \stackrel{\mathcal{D}}{\rightarrow} \text{Argmin}(Z) \text{ in } (\mathcal{F}, \tau_{F}).$$
\end{corollary}

\vspace{0.4cm}
In applications one is often more interested in single $\epsilon$-optimal solutions. So, for each $\alpha \in A$ let $\xi_\alpha: (\Omega_\alpha,\mathcal{A}_\alpha,\mathbb{P}_\alpha) \rightarrow (E,\mathcal{B(E)})$ be a measurable map. Such a random variable in $E$ is called \emph{measurable selection} of $A(Z_\alpha,\epsilon_\alpha)$ if $\xi_\alpha \in A(Z_\alpha,\epsilon_\alpha) \;
\mathbb{P}_\alpha$-almost-surely. Since $A(Z_\alpha,\epsilon_\alpha)$ is a random closed set by Corollary \ref{Aisrandomclosedset}, it follows from the Fundamental Selection Theorem, confer Molchanov \cite{Molchanov} on p. 77,
that measurable selections exist.\\

\begin{theorem} \label{mbselections} Assume that (\ref{Zeps}) holds and that for every $\eta>0$ there exists a compact $K \subseteq E$ with
$$
 \liminf_\alpha \mathbb{P}_\alpha(\{\omega \in \Omega_\alpha: \xi_\alpha(\omega) \in K\}) \ge 1-\eta.
$$
Then
\begin{equation} \label{PortmanteauforT}
 \limsup_\alpha  \mathbb{P}_\alpha(\{\omega \in \Omega_\alpha: \xi_\alpha(\omega) \in F\}) \le T(F) \quad \text{for all closed sets } F \text{ in }E,
\end{equation}
where $T:\mathcal{B}(E) \rightarrow [0,1]$ is the \emph{capacity functional} of $A(Z,\epsilon)$, i.e.
$$
 T(B)=\mathbb{P}(A(Z,\epsilon) \cap B \neq \emptyset) \quad \text{ for all Borel-sets } B \in \mathcal{B}(E).
$$
If $\epsilon=0$ and Argmin$(Z) \subseteq \{\xi\} \; \mathbb{P}$-almost surely for some random variable $\xi$ on $(\Omega,\mathcal{A},\mathbb{P})$ with values in $E$, then
$$
 \xi_\alpha \stackrel{\mathcal{D}}{\rightarrow} \xi \quad \text{in } (E,\mathcal{G}).
$$
\end{theorem}

\begin{proof} $A(Z_\alpha,\epsilon_\alpha) \stackrel{\mathcal{D}}{\rightarrow} A(Z,\epsilon)$ in $(\mathcal{F}, \tau_{uF})$ by Theorem \ref{epidistA}. Now, the assertion follows from Corollary 3.7
of Ferger \cite{Ferger2}.
\end{proof}

\begin{remark} \label{convergencetoset} Note that (\ref{PortmanteauforT}) looks exactly like the corresponding characterisation in the Portmanteau theorem, except for the fact that $T$ is generally not a probability measure, but only a Choquet-capacity.
On the other hand, $T$ uniquely determines the distribution of the random closed set $A(Z,\epsilon)$. It is therefore reasonable to say that
the \underline{points} $\xi_\alpha$ converge in distribution to the \underline{set} $A(Z,\epsilon)$.
\end{remark}

\vspace{0.3cm}
\begin{remark} \label{notequaltotheemptyset}
If we choose $F:=E \in \mathcal{F}$ in (\ref{PortmanteauforT}), then we obtain that $T(E)=1$. But $T(E)=\mathbb{P}(A(Z,\epsilon) \neq \emptyset)$, whence in particularly Argmin$(Z)=A(Z,0)$ is actually equal to $\{\xi\}$ with probability one.
\end{remark}

\section{Appendix}
Recall that $d$ is Vaughan's metric on $E$. Let $u$ be the usual euclidian distance on $\mathbb{R}$. We endow the product $E \times \mathbb{R}$ with the product-metric $d \times u$ defined
by $d \times u((x,\alpha),(y,\beta)):= \max\{d(x,y),|\alpha-\beta|\}$ for points $(x,\alpha)$ and $(y,\beta)$ in $E \times \mathbb{R}.$ It is well-known that this metric (among many others) induces the product-topology on $E \times \mathbb{R}$. The reason for our special choice lies in that
\begin{equation} \label{openballisopenrectangle}
 B_{d \times u}((x,\alpha),r)= B_d(x,r) \times (\alpha-r,\alpha+r),
\end{equation}
so open balls are open rectangles. Similarly,
\begin{equation} \label{closedballisclosedrectangle}
\overline{B}_{d \times u}((x,\alpha),r)= \overline{B}_d(x,r) \times [\alpha-r,\alpha+r].
\end{equation}
In particularly, $\overline{B}_{d \times u}((x,\alpha),r)$ is compact.\\

\begin{lemma} \label{continuityset} Let $P$ be a probability measure on $(S,\mathcal{B}_e)$ and $\phi$ be the homeomorphism (\ref{phi}). If $Q:= P \circ \phi^{-1}$, then for each $(x,\alpha) \in E \times \mathbb{R}$ and $r>0$ the pertaining closed ball
$\overline{B}_{d \times u}((x,\alpha),r)$ is a $Q$-continuity set, if and only if
$$
 P( I_{\overline{B}_d(x,r)} \le r+\alpha) = P( I_{(\overline{B}_d(x,r))^0} < r+\alpha).
$$
\end{lemma}

\begin{proof} We know that every closed ball $\overline{B}_{d \times u}((x,\alpha),r)$ is compact, whence by (\ref{Q-continuityset}) it is a $Q$-continuity set if and only if
\begin{eqnarray} \label{Pphicontinuityset}
 & &P(\{f \in S: \text{epi}(f) \in \mathcal{H}(\overline{B}_d(x,r) \times [\alpha-r,\alpha+r])\}) \nonumber\\
  &=& P(\{f \in S: \text{epi}(f) \in \mathcal{H}((\overline{B}_d(x,r) \times [\alpha-r,\alpha+r])^0)\}).
\end{eqnarray}
Now,
\begin{equation} \label{epi1}
 \text{epi}(f) \cap (\overline{B}_d(x,r) \times [\alpha-r,\alpha+r]) \neq \emptyset \; \Leftrightarrow \; I_{\overline{B}_d(x,r)}(f) \le r+\alpha.
\end{equation}
To see this, let $(y,t) \in E \times \mathbb{R}$ such that $t \ge f(y), y \in \overline{B}_d(x,r)$ and $t \in [\alpha-r,\alpha+r]$.
Then $I_{\overline{B}_d(x,r)}(f) \le f(y) \le t \le r+r.$ Conversely, assume that $I_{\overline{B}_d(x,r)}(f) \le r+s$. Since
$\overline{B}_d(x,r)$ is compact and $f$ is lsc, there exists some $z \in \overline{B}_d(x,r)$ with $f(z)=I_{\overline{B}_d(x,r)}(f).$
Consequently, $(z,r+\alpha) \in \text{epi}(f)$. Moreover, $(z,r+\alpha) \in \overline{B}_d(x,r) \times [\alpha-r,\alpha+r]$ as can be seen immediately.
Thus, we have shown (\ref{epi1}).
Further,
\begin{equation} \label{epi2}
 \text{epi}(f) \cap (\overline{B}_d(x,r) \times [\alpha-r,\alpha+r])^0 \neq \emptyset \; \Leftrightarrow \; I_{(\overline{B}_d(x,r))^0}(f) < r+\alpha.
\end{equation}
For the proof of necessity $\Rightarrow$, first observe that $(\overline{B}_d(x,r) \times [\alpha-r,\alpha+r])^0=(\overline{B}_d(x,r))^0 \times (\alpha-r,\alpha+r)$.
Let $(y,t) \in E \times \mathbb{R}$ with $t \ge f(y), y \in (\overline{B}_d(x,r))^0$ and $t \in (\alpha-r,\alpha+r)$. Then
$I_{(\overline{B}_d(x,r))^0)}(f) \le f(y) \le t < r+\alpha$.

As to sufficiency $\Leftarrow$, put $i:= I_{(\overline{B}_d(x,r))^0}(f).$
We know that $i< \alpha+r$. Then there exists some $u \in (\alpha-r,\alpha+r)$ with $i<u$, because otherwise $i \ge \alpha+r-\epsilon$ for all $\epsilon \in (0,2r)$ and taking the limit $\epsilon \downarrow 0$ yields that $i \ge \alpha+r$, a contradiction. Therefore, there exists some $v \in (\overline{B}_d(x,r))^0$ such that $f(v)< u$, whence
\begin{align*}
(v,u) \in \text{epi}(f) \cap (\overline{B}_d(x,r))^0 \times (\alpha-r,\alpha+r)\\
    = \text{epi}(f) \cap (\overline{B}_d(x,r) \times [\alpha-r,\alpha+r])^0 \neq \emptyset
\end{align*}
as desired.

From (\ref{Pphicontinuityset}) in combination with (\ref{epi1}) and (\ref{epi2}) the assertion follows.
\end{proof}

\vspace{1cm}
\textbf{Declarations}\\

\textbf{Compliance with Ethical Standards}: I have read and I understand the provided information.\\

\textbf{Competing Interests}: The author has no competing interests to declare that are relevant to the content of
this article.


\end{document}